\documentclass[12pt]{amsart}
\usepackage{amsfonts,amsthm,amsmath,amscd,amssymb}
\usepackage[T1]{fontenc}
\usepackage[mathscr]{eucal}
\usepackage{indentfirst}
\usepackage{graphicx}
\usepackage{graphics}
\usepackage{xcolor}
\usepackage{pict2e}
\usepackage{epic}
\usepackage{hyperref}
\usepackage{mathtools}
\numberwithin{equation}{section}
\usepackage[margin=2.9cm]{geometry}
\usepackage{epstopdf} 
\usepackage{cite}
\usepackage{parskip}
\usepackage{array}
\usepackage{booktabs}
\usepackage{enumerate}
\usepackage{adjustbox}
\usepackage{diagbox}
\usepackage{tikz,tikz-3dplot}
\tdplotsetmaincoords{80}{65}

\tikzset{surface/.style={draw=black!60!,thick}}

\newcommand{\coneback}[4][]{
  \draw[canvas is xy plane at z=#2, #1] (\tdplotmainphi-#4:#3) 
  arc(\tdplotmainphi-#4:\tdplotmainphi+180+#4:#3) -- (O) --cycle;
  }
\newcommand{\conefront}[4][]{
  \draw[canvas is xy plane at z=#2, #1] (\tdplotmainphi-#4:#3) arc
  (\tdplotmainphi-#4:\tdplotmainphi-180+#4:#3) -- (O) --cycle;
  }

\theoremstyle{plain}
\newtheorem{thm}{Theorem}[section]
\newtheorem{lem}[thm]{Lemma}

\newtheorem{prop}[thm]{Proposition}

 \theoremstyle{definition}
\newtheorem{defn}[thm]{Definition}
\newtheorem{rem}[thm]{Remark}
\newtheorem{notn}[thm]{Notation}

\newtheorem{setup}[thm]{Setup}

\newcommand{\Hom}{{\rm{Hom}}}

\newcommand{\ovl}{\overline}

\newcommand{\bm}[1]{\mathbf{#1}}

\newcommand{\mb}[1]{\mathbb{#1}}
\newcommand{\mc}[1]{\mathcal{#1}}

\newcommand{\mr}[1]{\mathrm{#1}}

\newcommand{\vphi}{\varphi}

\newcommand{\Char}{\operatorname{char}}

\newcommand{\Spec}{\operatorname{Spec}}

\newcommand{\GW}{\operatorname{GW}}

\newcommand{\Sym}{\operatorname{Sym}}
\newcommand{\Tr}{\operatorname{Tr}}
\newcommand{\ind}{\operatorname{ind}}
\newcommand{\sign}{\operatorname{sign}}
\newcommand{\disc}{\operatorname{disc}}
\newcommand{\rank}{\operatorname{rank}}
\newcommand{\Jac}{\operatorname{Jac}}
\newcommand{\Hilb}{\operatorname{Hilb}}
\newcommand{\codim}{\operatorname{codim}}
\newcommand{\pcoor}[1]{%
  \begingroup\lccode`~=`: \lowercase{\endgroup
  \edef~}{\mathbin{\mathchar\the\mathcode`:}\nobreak}%
  [
  \begingroup
  \mathcode`:=\string"8000
  #1%
  \endgroup 
  ]
}

\makeatletter
\@namedef{subjclassname@2020}{\textup{2020} Mathematics Subject Classification}
\makeatother

\begin{document}
\title{Conics meeting eight lines over perfect fields}

\author[Darwin]{Cameron Darwin}
\address{Department of Mathematics, Duke University} 
\email{cameron.darwin@duke.edu}

\author[Galimova]{Aygul Galimova}
\address{Department of Mathematics, Duke University} 
\email{aygul.galimova@duke.edu}

\author[Gu]{Miao (Pam) Gu}
\address{Department of Mathematics, Duke University}  
\email{pmgu@math.duke.edu}
\urladdr{sites.duke.edu/pmgu/}

\author[McKean]{Stephen McKean}
\address{Department of Mathematics, Harvard University} 
\email{smckean@math.harvard.edu}
\urladdr{shmckean.github.io}
\begin{abstract}
    Over the complex numbers, there are 92 plane conics meeting 8 general lines in projective 3-space. Using the Euler number and local degree from motivic homotopy theory, we give an enriched version of this result over any perfect field. This provides a weighted count of the number of plane conics meeting 8 general lines, where the weight of each conic is determined the geometry of its intersections with the 8 given lines. As a corollary, real conics meeting 8 general lines come in two families of equal size.
\end{abstract}
\maketitle
\section{Introduction}
The space of plane conics in $\mb{P}^3$ is 8 dimensional. If we require that a conic intersects a given line, we impose one condition and lose one degree of freedom on the space of plane conics. As a result, the space of plane conics meeting 8 general lines is a 0 dimensional Noetherian scheme and is therefore a finite set. A classical theorem of enumerative geometry gives the cardinality of this set.

\begin{thm}\label{thm:main-classical}
Let $k$ be an algebraically closed field with $\Char{k}\neq 2$. Given 8 lines in general position in $\mb{P}^3_k$, there are 92 plane conics meeting all 8 lines. Moreover, each of these plane conics is smooth. (See e.g.~\cite[Ex. 3.2.22]{Ful98} or~\cite[Theorem 9.26]{eisenbud_harris_2016}.)
\end{thm}

As with many results in classical enumerative geometry, this theorem is only true over an algebraically closed field. The \textit{$\mb{A}^1$-enumerative geometry} program seeks to generalize such theorems using various tools from motivic homotopy theory.\footnote{See \cite{KW21,Lev20,RL20,BKW20,SW21,LV19,McK21,Pau20,CDH20} for some examples or \cite{Bra20,PW20} for a survey.} We give an $\mb{A}^1$-enumerative generalization of Theorem~\ref{thm:main-classical}.

We start with some notation. Let $k$ be a perfect field with $\Char{k}\neq 2$. Given a conic $q\subset\mb{P}^3_k$, let $k(q)$ be its field of definition. Let $\GW(k)$ be the Grothendieck--Witt group of isomorphism classes of symmetric, non-degenerate bilinear forms over $k$. Given $a\in k^\times$, let $\langle a\rangle\in\GW(k)$ be the bilinear form given by $(x,y)\mapsto axy$. Finally, let $\Tr_{k(q)/k}:\GW(k(q))\to\GW(k)$ be induced by the field trace.

\begin{thm}\label{thm:main-enriched}
Let $L_1,\ldots,L_8$ be lines in general position in $\mb{P}^3_k$. Let $Q$ be the set of all plane conics in $\mb{P}^3_k$ meeting $L_1,\ldots,L_8$. Then
\begin{align}\label{eq:main}
    46\langle 1\rangle+46\langle -1\rangle=\sum_{q\in Q}\Tr_{k(q)/k}\langle a_q\rangle,
\end{align}
where $a_q\in k(q)^\times$ is a constant determined by the conic $q$, the intersections $L_i\cap q$, and the tangent lines $T_{L_i\cap q}q$ for $1\leq i\leq 8$.
\end{thm}

Theorem~\ref{thm:main-enriched} gives some insight into real conics meeting 8 general lines. Hauenstein and Sottile showed that over $\mb{R}$, there can be $2n$ real conics meeting 8 lines for $0\leq n\leq 45$ \cite[Table 6]{HS12}. Griffin and Hauenstein completed this result by constructing 8 lines over $\mb{R}$ such that all 92 conics meeting these lines are real \cite[Theorem 1]{GH15}. Our work illuminates a small amount of extra structure on this set of $2n$ conics. Taking the signature of Equation~\ref{eq:main} yields Theorem~\ref{thm:real}, which states that the $2n$ real conics meeting 8 general lines fall into two families of $n$ conics.

\subsection{General approach and outline}\label{sec:approach}
Our goal is to prove an equality in $\GW(k)$, the Grothendieck--Witt group of isomorphism classes of non-degenerate symmetric bilinear forms over $k$. One side of this equation will be given by an Euler number \cite{KW21,BW20}, which is valued in $\GW(k)$ in the context of motivic homotopy theory. The other side of this equation will consist of a sum of local contributions, which are analogs of the local Brouwer degree \cite{Mor12,KW19,KW21}. The final step is to find a formula for these local contributions in terms of the geometry at hand -- in our case, the geometry of lines meeting a plane conic.

In order to make use of Euler numbers and local degrees, we need to phrase our enumerative problem in terms of a vector bundle over a scheme parameterizing conics in $\mb{P}^3$. This has been done classically~\cite[Chapter 9.7]{eisenbud_harris_2016}. We will recall the relevant details here.

\subsubsection*{Terminology}
Typically, a line, plane, or conic over a field $k$ refer to these objects as varieties over $k$. We will only work with lines defined over the base field $k$ in this article. However, we will work with planes and conics over finite extensions of $k$. These planes and conics arise as closed points in their parameter spaces. For example, a point $H\in\mb{G}(2,3)$ in the Grassmannian of 2-planes in $\mb{P}^3_k$ represents a 2-plane in $\mb{P}^3_{k'}$, where $k'$ is the residue field of $H$. Similarly, a point $(H,q)$ in the Hilbert scheme of conics represents a degree 2 curve defined over $k'$ contained in a 2-plane in $\mb{P}^3_{k'}$, where $k'$ is the residue field of $(H,q)$. We may thus refer to the residue field of $(H,q)$ as the \textit{field of definition} of the conic represented by $(H,q)$.

\subsubsection*{Space of conics}
Let $k$ be a field. The Hilbert scheme $\Hilb_{2t+1}(\mb{P}^3_k)$ is the moduli scheme parameterizing conics in $\mb{P}^3_k$. However, it will be more convenient for us to work with a different presentation of this moduli space. Any subscheme of $\mb{P}^3_k$ with Hilbert polynomial $2t+1$ is the complete intersection of a plane and a quadric surface. If the conic is reduced, this plane is uniquely determined by three non-colinear points on the conic. If the conic is a double line defined over $k$, then we may use a $k$-linear change of coordinates such that the support of the double line is $\{x_0=x_1=0\}$. By standard considerations on the Hilbert polynomial, we deduce that the double line must have defining ideal $(x_0^2,x_0x_1,x_1^2,x_0f(x_2,x_3)+x_1g(x_2,x_3))$, where $f,g\in k[x_2,x_3]$ are in fact constants. It follows that the plane $\mb{V}(x_0f+x_1g)$ is again uniquely determined by the conic. Putting these cases together, we get a morphism $\Hilb_{2t+1}(\mb{P}^3_k)\to\mb{G}(2,3)$, where $\mb{G}(2,3)$ is the Grassmannian of 2-planes in $\mb{P}^3_k$. The fiber of this map is the space of conics in the plane, namely $\Hilb_{2t+1}(\mb{P}^2_k)$.

Let $\mc{S}$ be the universal subbundle of $\mb{G}(2,3)$. Consider the symmetric bundle $\Sym^2(\mc{S}^\vee)$ of planar quadratic forms, which is a rank 6 vector bundle over $\mb{G}(2,3)\cong\mb{P}^3_k$. The points of the projective bundle $\mb{P}\Sym^2(\mc{S}^\vee)\to\mb{G}(2,3)$ correspond to projective classes of homogeneous quadratic polynomials on planes in $\mb{P}^3_k$. By the universal property of Hilbert schemes, we thus get a morphism $\phi:\mb{P}\Sym^2(\mc{S}^\vee)\to\Hilb_{2t+1}(\mb{P}^3_k)$. Moreover, since each conic in $\mb{P}^3_k$ uniquely determines its plane, the map $\phi$ is a bijection of points. Finally, since Hilbert schemes respect base change and $\Hilb_{2t+1}(\mb{P}^3_{\ovl{k}})$ is smooth and irreducible, it follows that $\Hilb_{2t+1}(\mb{P}^3_k)$ is also smooth and irreducible~\cite[\href{https://stacks.math.columbia.edu/tag/05B5}{Lemma 05B5} and \href{https://stacks.math.columbia.edu/tag/038I}{Lemma 038I}]{stacks}. We now conclude by Zariski's main theorem that $\phi$ is an isomorphism. We will use $X:=\mb{P}\Sym^2(\mc{S}^\vee)\to\mb{G}(2,3)$ as our presentation of the moduli space of conics in $\mb{P}^3_k$. The rational points of $X$ are of the form $(H,q)$, where $H\in\mb{G}(2,3)$ is a plane and $q\in\mb{P}\Sym^2(H^\vee)$ is the projective class of a homogeneous quadratic polynomial on $H$ (whose vanishing defines a conic on $H$).

\subsubsection*{Conics meeting a line}
Next, we need a vector bundle on $X$ with a global section that vanishes precisely on conics that meet a given line. We will define such a (line) bundle and section on an open subset $U\subseteq X$ such that $\codim(X\backslash U)\geq 2$. Since $X$ is smooth, $X$ satisfies Serre's $S_2$-criterion for extending coherent sheaves. In particular, we can extend coherent sheaves over $U$ to coherent sheaves over $X$. In order to promote such an extension of coherent sheaves to an extension of vector bundles, we would need to verify that the coherent sheaf on $X$ extending our vector bundle on $U$ is locally free of finite rank. This is automatic when the bundle on $U$ is a line bundle, which is the case at hand.

Moreover, since $X$ is a projective bundle over a smooth projective $k$-scheme, $X$ is itself a smooth projective $k$-scheme. One can thus show that $\mr{depth}_{\mc{I}}(\mc{L})\geq 2$ for any line bundle $\mc{L}\to X$, where $\mc{I}$ is the ideal of the closed complement $X\backslash U$. It follows that
\[H^0_{X\backslash U}(X,\mc{L})=H^1_{X\backslash U}(X,\mc{L})=0,\]
so the long exact sequence
\[\cdots\to H^i_{X\backslash U}(X,\mc{L})\to H^i(X,\mc{L})\to H^i(U,\mc{L}|_U)\to H^{i+1}_{X\backslash U}(X,\mc{L})\to\cdots\]
yields an isomorphism $H^0(X,\mc{L})\cong H^0(U,\mc{L}|_U)$. In particular, we can extend global sections of line bundles over codimension 2 subsets in $X$.

Let $L\in\mb{G}(1,3)$. The locus of planes $H\in\mb{G}(2,3)$ such that $L\subset H$ forms a pencil $P\subset\mb{G}(2,3)$ of dimension 1. Thus the locus $U:=\pi^{-1}(\mb{G}(2,3)\backslash P)\subseteq X$ is an open subset (and hence subscheme) whose complement is of codimension at least 2. At the level of points, $U$ consists of pairs $(H,q)$, where $H$ is a plane not containing $L$ and $q$ is the projective class of a homogeneous quadratic polynomial on $H$. 

Consider the morphism $\alpha:U\to L$ given by $L\cap\pi|_U(-)$. At the level of points, we have $\alpha(H,q)=L\cap H$. Vanishing at a point imposes a linear condition on the space of plane conics, so we have a subbundle $Q_L\subseteq\mb{P}\Sym^2(\mc{S}^\vee)$ whose fiber over $H\in\pi(U)$ is the space $Q(H)_{L\cap H}\cong\mb{P}^4_k$ of projective classes of homogeneous quadratic polynomials on $H$ that vanish at $L\cap H$. By construction, we have a short exact sequence
\[0\to Q_L\to\mc{O}_U(-1)\to\alpha^*\mc{O}_L(2)\to 0.\]
It follows that the bundle morphism $\mc{O}_U(-1)\to\alpha^*\mc{O}_L(2)$ is given by the evaluation map $\mr{ev}_L(H,q)\mapsto q\text{ mod }Q(H)_{L\cap H}$. Equivalently, the zero section $z\in H^0(U,\mc{O}_U(-1))$ and the evaluation map $\mr{ev}_L:\mc{O}_U(-1)\to\alpha^*\mc{O}_L(2)$ yield a global section $\mr{ev}_L\circ z$ of 
\[\Hom(\mc{O}_U(-1),\alpha^*\mc{O}_L(2))\cong\mc{O}_U(1)\otimes\alpha^*\mc{O}_L(2)\] 
that vanishes precisely on conics that intersect $L$. On points, this section is given by the formula $\mr{ev}_L\circ z(H,q)=q(L\cap H)$. By our assumption that $L$ meets every plane in $\pi(U)$ transversely, we have an isomorphism of bundles 
\[\alpha^*\mc{O}_L(2)\cong(\pi|_U)^*\mc{O}_{\mb{G}(2,3)}(2).\]
Extending this line bundle and the section $\mr{ev}_L\circ z$ across the complement of $U$, we obtain the line bundle $E:=\mc{O}_X(1)\otimes\pi^*\mc{O}_{\mb{G}(2,3)}(2)\to X$ and global section $\sigma_L$. Given 8 general lines $L_1,\ldots,L_8\in\mb{G}(1,3)$, the section $\sigma:=\bigoplus_{i=1}^8\sigma_{L_i}:X\to E^{\oplus 8}$ vanishes precisely on conics that meet each $L_1,\ldots,L_8$.

\subsubsection*{$\mb{A}^1$-enumerative count}
The classical count of conics meeting 8 general lines in $\mb{P}^3_\mb{C}$ is given by $\int_X c_1(E)^8=92$~\cite[Section 9.7.3]{eisenbud_harris_2016}. The enriched count of conics meeting 8 general lines over a field $k$ is given by the Euler number $e(E^{\oplus 8})\in\GW(k)$, which we can compute using a result of Srinivasan and Wickelgren~\cite[Proposition 19]{SW21}. This Euler number is equal to a sum of local information over the set of conics $(H,q)$ meeting $L_1,\ldots,L_8$ \cite[Theorem 3]{KW21}:
\begin{equation}\label{eq:sum of degrees}
    e(E^{\oplus 8})=\sum_{(H,q)\in\sigma^{-1}(0)}\ind_{(H,q)}(\sigma).
\end{equation}
After computing the Euler number in Lemma~\ref{lem:euler class}, we will address the local indices $\ind_{(H,q)}(\sigma)$. Given a conic $(H,q)$ in an affine neighborhood $U\subset X$, we will give invertible Nisnevich coordinates $\vphi_U^{-1}:U\to\mb{A}^8_k$ and local trivializations (post-composed with projection) $\psi_U:E^{\oplus 8}|_U\to U\times\mb{A}^8_k\to\mb{A}^8_k$ in Section~\ref{sec:coords}. The local index $\ind_{(H,q)}(\sigma)$ is equal to the local $\mb{A}^1$-degree $\deg^{\mb{A}^1}_{\vphi_U^{-1}(H,q)}(\Phi_U)$ of the composite
\[\Phi_U:=\psi_U\circ\sigma\circ\vphi_U:\mb{A}^8_k\to\mb{A}^8_k\]
at a zero $(H,q)\in U$. In Section~\ref{sec:geometric interpretation}, we give an alternate formula for $\deg^{\mb{A}^1}_{\vphi_U^{-1}(H,q)}(\Phi_U)$ in terms of the intersection points $L_i\cap H$ and tangent lines of the conic $\mb{V}(q)\cap H$.
By replacing $\ind_{(H,q)}(\sigma)$ in Equation~\ref{eq:sum of degrees} with our alternate formula in terms of geometric information, we recover Theorem~\ref{thm:main-enriched}. We conclude with a discussion of real conics in Section~\ref{sec:real}.

\subsection{Acknowledgements}
We thank Kirsten Wickelgren for suggesting this problem during her topics course on intersection theory, as well as for her support and feedback. We also thank Yupeng Li for his collaboration on an early stage of this project. We thank Ben Williams for helpful comments. The last named author thanks Thomas Brazelton and Sabrina Pauli for helpful discussions about conics in $\mb{P}^3$. Finally, we thank the anonymous referee, whose thorough reading and suggestions greatly improved our exposition and corrected several of our mistakes. 

The third named author received support from Jayce Getz's NSF grant (DMS-1901883). The last named author received support from Kirsten Wickelgren's NSF CAREER grant (DMS-1552730) and an NSF MSPRF (DMS-2202825).

\section{Background in $\mb{A}^1$-enumerative geometry}\label{sec:background}
In classical enumerative geometry, one is interested in (possibly weighted) integer-valued counts of geometric objects. For example, if $Q$ is the set of plane conics meeting 8 general lines in $\mb{P}^3_\mb{C}$, then
\begin{align}\label{eq:classical}
92=\sum_{q\in Q}1.
\end{align}
In $\mb{A}^1$-enumerative geometry, we replace such integer-valued counts with bilinear form-valued counts. We will show that if $Q$ is the set of plane conics meeting 8 general lines in $\mb{P}^3_k$ over a perfect field $k$, then
\begin{align}\label{eq:a1-enumerative}
46\langle 1\rangle+46\langle -1\rangle=\sum_{q\in Q}B_q.
\end{align}
Here, the bilinear form $\langle a\rangle:k\times k\to k$ is given by $(x,y)\mapsto axy$. The weight $B_q$ is a bilinear form determined by geometric information associated to the plane conic $q$ (see Section~\ref{sec:geometric interpretation}). By taking field invariants, we can recover enumerative equations over specific fields. For example, taking the rank of Equation~\ref{eq:a1-enumerative} recovers Equation~\ref{eq:classical}, while taking the signature of Equation~\ref{eq:a1-enumerative} yields a new theorem (Theorem~\ref{thm:real}) giving a weighted count of conics meeting 8 general lines over $\mb{R}$.

\subsection{Grothendieck--Witt groups}
The significance of bilinear forms in $\mb{A}^1$-enumerative geometry stems from Morel's calculation of the Brouwer degree in $\mb{A}^1$-homotopy theory (also known as motivic homotopy theory):
\begin{thm}\cite[Corollary 1.24]{Mor12}
For $n\geq 2$, there is a group (and in fact, ring) isomorphism 
\[
\deg^{\mb{A}^1}:[\mb{P}^n_k/\mb{P}^{n-1}_k,\mb{P}^n_k/\mb{P}^{n-1}_k]_{\mb{A}^1}\xrightarrow{\cong}\GW(k),
\]
where $[-,-]_{\mb{A}^1}$ denotes $\mb{A}^1$-homotopy classes of maps, $\mb{P}^n_k/\mb{P}^{n-1}_k$ is a motivic space playing the role of the sphere, and $\GW(k)$ is the Grothendieck--Witt group of isomorphism classes of symmetric non-degenerate bilinear forms over $k$.
\end{thm}

Morel's degree map is analagous to the Brouwer degree
\[\deg:[S^n,S^n]\xrightarrow{\cong}\mb{Z}.\]
One can apply Morel's degree to endomorphisms of $\mb{A}^n_k$ to obtain bilinear forms. The goal of $\mb{A}^1$-enumerative geometry is to perform this process in such a way that the resulting bilinear forms encode enumerative information. Later in this section, we will discuss how to circumvent explicitly using motivic homotopy theory in $\mb{A}^1$-enumerative geometry. First, we briefly discuss $\GW(k)$. The Grothendieck--Witt group $\GW(k)$ is actually a ring that admits a nice presentation.

\begin{prop}\label{prop:GW}\cite[II Theorem 4.1]{Lam05}
Let $k$ be a field. Given $a\in k^\times$, let $\langle a\rangle$ be the isomorphism class of the bilinear form $k\times k\to k$ defined by $(x,y)\mapsto axy$. Then $\GW(k)$ is the ring generated by all such $\langle a\rangle$, subject to the following relations.
\begin{enumerate}[(i)]
\item $\langle ab^2\rangle=\langle a\rangle$ for all $a,b\in k^\times$.
\item $\langle a\rangle\langle b\rangle=\langle ab\rangle$ for all $a,b\in k^\times$.
\item $\langle a\rangle+\langle b\rangle=\langle a+b\rangle+\langle ab(a+b)\rangle$ for all $a,b\in k^\times$ such that $a+b\neq 0$.
\item $\langle a\rangle+\langle -a\rangle=\langle 1\rangle+\langle -1\rangle$ for all $a\in k^\times$.
\end{enumerate}
\end{prop}

Relation (iv) actually follows from relations (i) and (iii). Indeed, (iii) implies that $\langle -a\rangle+\langle a-1\rangle=\langle a^2-a\rangle+\langle -1\rangle$ and that $\langle a\rangle+\langle a^2-a\rangle=\langle a^2\rangle+\langle a^2(a-1)\rangle$. Thus by (i), we have
\begin{align*}
    \langle a\rangle+\langle -a\rangle&=\langle a\rangle+\langle a^2-a\rangle+\langle -1\rangle-\langle a-1\rangle\\
    &=\langle a^2\rangle+\langle a^2(a-1)\rangle+\langle -1\rangle-\langle a-1\rangle\\
    &=\langle 1\rangle+\langle a-1\rangle+\langle -1\rangle-\langle a-1\rangle\\
    &=\langle 1\rangle+\langle -1\rangle.
\end{align*}

\begin{defn}
The isomorphism class $\mb{H}:=\langle 1\rangle+\langle -1\rangle$ is called the \textit{hyperbolic form}.
\end{defn}

In order to obtain enumerative statements over a given field, we apply \textit{field invariants} to our enumerative equation in $\GW(k)$. Field invariants can be thought of as group homomorphisms $\GW(k)\to G$ for some group $G$. For example:
\begin{enumerate}[(i)]
\item The rank of a bilinear form induces an isomorphism $\rank:\GW(\mb{C})\to\mb{Z}$.
\item The signature of a bilinear form (the number of $+1$s minus the number of $-1$s on the diagonal) induces a homomorphism $\sign:\GW(\mb{R})\to\mb{Z}$.
\item The discriminant of a bilinear form induces a homomorphism $\disc:\GW(\mb{F}_q)\to\mb{Z}/2\mb{Z}$ when $q$ is a power of an odd prime.
\end{enumerate}

See~\cite{Lam05} for a discussion on field invariants for various fields. See~\cite{KW21,LV19,McK21,CDH20} for examples of applying these field invariants to obtain enumerative statements over specific fields.

\subsection{Local $\mb{A}^1$-degrees}\label{sec:local degree}
Many results in $\mb{A}^1$-enumerative geometry are an application of the Poincar\'e--Hopf theorem for motivic Euler numbers. The Euler number computes a fixed element of $\GW(k)$, which constitutes the global count of objects in question. The Poincar\'e--Hopf theorem then states that this fixed value can be expressed as a sum of local indices. In general, these local indices are not fixed, but rather record some of the arithmetic and geometry of the specific objects being counted. Because these local indices will be defined in terms of the \textit{local $\mb{A}^1$-degree}, we use this section to quickly survey the relevant construction. See also~\cite{KW19} and~\cite[Section 4.2]{WW20}.

Let $f:\mb{A}^n_k\to\mb{A}^n_k$ be a morphism with an isolated zero $p\in\mb{A}^n_k$. This induces a map
\[f_p:\mb{A}^n_k/(\mb{A}^n_k-\{p\})\to\mb{A}^n_k/(\mb{A}^n_k-\{0\})\]
in the pointed unstable motivic homotopy category $\mc{H}_\bullet(k)$. The source and target of $f_p$ are Thom spaces in $\mc{H}_\bullet(k)$, and the purity theorem of Morel and Voevodsky~\cite{MV99} gives weak equivalences
\[\mb{A}^n_k/(\mb{A}^n_k-\{p\})\simeq\mb{P}^n_k/\mb{P}^{n-1}_k\wedge\Spec{k(p)}_+\quad\text{and}\quad\mb{A}^n_k/(\mb{A}^n_k-\{0\})\simeq\mb{P}^n_k/\mb{P}^{n-1}_k.\]
If $p$ is $k$-rational, then we thus have a map
\[\mb{P}^n_k/\mb{P}^{n-1}_k\simeq\mb{A}^n_k/(\mb{A}^n_k-\{p\})\xrightarrow{f_p}\mb{A}^n_k/(\mb{A}^n_k-\{0\})\simeq\mb{P}^n_k/\mb{P}^{n-1}_k\]
to which we can apply Morel's $\mb{A}^1$-degree. In general, we precompose with the collapse map $c_p:\mb{P}^n_k/\mb{P}^{n-1}_k\to\mb{P}^n_k/\mb{P}^{n-1}_k\wedge\Spec{k(p)}_+$, which is defined geometrically via the inclusion $\mb{P}^{n-1}_k\hookrightarrow\mb{P}^n_k-\{p\}$, yielding the diagram of cofibers $\mb{P}^n_k/\mb{P}^{n-1}_k\to\mb{P}^n_k/(\mb{P}^n_k-\{p\})$.

\begin{defn}
The \textit{local $\mb{A}^1$-degree} of a morphism $f:\mb{A}^n_k\to\mb{A}^n_k$ at an isolated zero $p$, denoted $\deg_p(f)\in\GW(k)$, is the $\mb{A}^1$-degree of the composite
\[\mb{P}^n_k/\mb{P}^{n-1}_k\xrightarrow{c_p}\mb{P}^n_k/\mb{P}^{n-1}_k\wedge\Spec{k(p)}_+\simeq\mb{A}^n_k/(\mb{A}^n_k-\{p\})\xrightarrow{f_p}\mb{A}^n_k/(\mb{A}^n_k-\{0\})\simeq\mb{P}^n_k/\mb{P}^{n-1}_k.\]
\end{defn}

Despite its technical definition in terms of motivic homotopy theory, the local $\mb{A}^1$-degree admits a convenient commutative algebraic formulation \cite{BMP21} (see also \cite{KW19,BBMMO21}). For our purposes, we will be able to compute all relevant local $\mb{A}^1$-degrees in terms of the Jacobian by \cite[Lemma 9]{KW19} (which stems from \cite[(4.7) Korollar]{SS75}).

\subsection{Euler numbers}\label{sec:Euler number}
The Euler numbers that we work with were introduced in \cite{KW21} and further studied in \cite{BW20}. In this section, we will recall the definition of these Euler numbers. See \cite[Section 1.1]{KW21} and the introduction of \cite{BW20} for a discussion of related notions of Euler classes and numbers in arithmetic geometry and motivic homotopy theory.

The most refined definition of the Bachmann--Kass--Wickelgren Euler number (hereafter, simply \textit{Euler number}) is given via coherent duality. We will follow \cite[Section 2.1]{BW20} in our review of the details. To begin, we need to recall the notion of relative orientation of a vector bundle.

\begin{defn}
Let $X$ be a $k$-scheme, and let $V\to X$ be a vector bundle. A \textit{relative orientation} of $V\to X$ is a line bundle $\mc{L}\to X$ and an isomorphism $\rho:\Hom(\det{V}^\vee,\omega_X)\xrightarrow{\cong}\mc{L}^{\otimes 2}$. We say that $V\to X$ is \textit{relatively orientable} if a relative orientation of $V\to X$ exists.
\end{defn}

Now suppose that $X$ is a smooth proper $k$-scheme of dimension $n$ with structure map $f:X\to\Spec{k}$. Let $V\to X$ be a relatively orientable vector bundle of rank $n$ with relative orientation $(\rho,\mc{L})$. This relative orientation defines an isomorphism $\rho':\det{V}^\vee\otimes\mc{L}^{\otimes 2}\to\omega_X$, and coherent duality defines a trace map $\eta_f:H^n(X,\omega_X)\to k$. For each $0\leq i,j\leq n$, we thus get a perfect pairing $\beta_{i,j}:H^i(X,\wedge^jV^\vee\otimes\mc{L})\otimes H^{n-i}(X,\wedge^{n-j}V^\vee\otimes\mc{L})\to k$ given by the composition
\[H^i(X,\wedge^jV^\vee\otimes\mc{L})\otimes H^{n-i}(X,\wedge^{n-j}V^\vee\otimes\mc{L})\xrightarrow{\smile}H^n(X,\det{V}^\vee\otimes\mc{L}^{\otimes 2})\xrightarrow{\rho'}H^n(X,\omega_X)\xrightarrow{\eta_f}k.\]
On the central term (i.e. when $2i=2j=n$), $\beta_{i,j}$ is a bilinear form on $H^i(X,\wedge^jV^\vee\otimes\mc{L})$. For all other $i,j$, the pairing $\beta_{i,j}\oplus\beta_{n-i,n-j}$ is a bilinear form on $H^i(X,\wedge^jV^\vee\otimes\mc{L})\oplus H^{n-i}(X,\wedge^{n-j}V^\vee\otimes\mc{L})$. It follows that $\sum_{0\leq i,j\leq n}(-1)^{i+j}\beta_{i,j}$ is a symmetric, non-degenerate bilinear form over $k$.

\begin{defn}\label{def:Euler number}
The \textit{Euler number} $e(V)\in\GW(k)$ of the vector bundle $V\to X$ is the isomorphism class of the bilinear form $\sum_{0\leq i,j\leq n}(-1)^{i+j}\beta_{i,j}$.
\end{defn}

\subsection{Local indices}
In classical algebraic topology, a powerful aspect of Euler numbers is the Poincar\'e--Hopf theorem. This is a local-to-global principle, which states that variable local behavior (as measured by local indices) is governed by a fixed global invariant, namely the Euler number. As proved in~\cite[Theorem 1.1]{BW20}, the Poincar\'e--Hopf theorem also holds for these motivic Euler numbers. In this section, we will describe and discuss local indices for the Euler number. These local indices will be elements of $\GW(k)$ determined by the local behavior of a given section $\sigma$ along its vanishing locus.

To begin, we need a suitable notion of local coordinates, as well a local trivialization that are compatible with these coordinates in a precise way. These coordinates and trivialization will allow us to turn a section $\sigma:X\to V$ of a rank $n$ vector bundle on an $n$-dimensional scheme into an endomorphism $f:\mb{A}^n_k\to\mb{A}^n_k$. We will then compute the local index of $\sigma$ at a zero $p\in X$ by computing an analog of the local Brouwer degree of $f$ at the image of $p$. These ideas and tools were introduced in~\cite{KW19,KW21}.

\begin{defn}
Let $X$ be a smooth $k$-scheme of dimension $n$. Let $p\in X$ be a closed point. \textit{Nisnevich coordinates} around $p$ consist of a Zariski open neighborhood $U\subseteq X$ containing $p$ and an \'etale morphism $\vphi:U\to\mb{A}^n_k$ that induces an isomorphism $k(p)\cong k(\vphi(p))$ of residue fields.
\end{defn}

By~\cite[Lemma 19]{KW21}, there are Nisnevich coordinates around any closed point on any smooth $k$ scheme of dimension at least 1.

As previously mentioned, we will define the local index of a section in terms of the local degree of a morphism determined by the section and a choice of a local trivialization and Nisnevich coordinates. In order to ensure that this local index does not depend on our choice of coordinates or trivialization, we have to impose the following compatibility condition. Note that this compatibility also involves the relative orientation used to define the Euler number of our vector bundle, as outlined in Section~\ref{sec:Euler number}.

\begin{setup}\label{setup:compatibility}
Let $X$ be a smooth $k$-scheme of dimension $n$. Let $(U,\vphi)$ be Nisnevich coordinates around a closed point $p\in X$. Since $\vphi$ is \'etale, the pullback by $\vphi$ of the standard basis of $T\mb{A}^n_k$ is a basis of $TX|_U$. We denote the dual of this basis by $d\vphi$, whose determinant $\det(d\vphi)$ is a basis of $\det(TX|_U^\vee)=\omega_X|_U$.

Now let $V\to X$ be a relatively orientable vector bundle with relative orientation $(\rho,\mc{L})$. Let $\psi:V|_U\xrightarrow{\cong} U\times\mb{A}^n_k\xrightarrow{\pi_2}\mb{A}^n_k$ be a local trivialization of $V$, followed by projection away from the base. We can view $\psi$ as a basis of $V|_U$ (ignoring the base $U$), so $\det(\psi)$ is a basis of $\det(V|_U)$.

Finally, let $\eta\in\Hom(\det{V}|_U^\vee,\omega_X|_U)$ be the homomorphism defined by $\eta(\det(\psi)^\vee)=\det(d\vphi)$. That is, $\eta$ is the map sending the distinguished basis $\det(\psi)^\vee$ of $\det(V|_U)^\vee$ (determined by the trivialization $\psi$) to the distinguished basis $\det(d\vphi)$ of $\omega_X|_U$ (determined by the Nisnevich coordinates $(U,\vphi)$).
\end{setup}

\begin{defn}\label{def:compatibility}
Assume the notation of Setup~\ref{setup:compatibility}.  The local trivialization $\psi$ of $V|_U$ is said to be \textit{compatible} with the relative orientation $(\rho,\mc{L})$ and the Nisnevich coordinates $(U,\vphi)$ if $\rho(\eta)=\ell\otimes\ell$ for some $\ell\in\mc{L}$.
\end{defn}

Roughly speaking, the compatibility condition given in Definition~\ref{def:compatibility} states that changing our choice of Nisnevich coordinates, local trivialization, or relative orientation, provided that these data are compatible with each other, will only change the local index of a section by a square. Our local index will be valued in $\GW(k)$, so such squares will be trivial by Proposition~\ref{prop:GW} (i). By~\cite[Proposition 5.5]{McK22}, we can always find a local trivialization compatible with our chosen Nisnevich coordinates.

Heuristically, we define the local index as follows. Let $V\to X$ be a relatively orientable vector bundle with a section $\sigma:X\to V$. Let $p\in X$ be an isolated zero of $\sigma$, and let $(U,\vphi)$ be Nisnevich coordinates around $p$. Finally, let $\psi:V|_U\to\mb{A}^n_k$ be a local trivialization that is compatible with the given Nisnevich coordinates and some given relative orientation of $V\to X$. If $\vphi:U\to\mb{A}^n_k$ were an isomorphism, we could form the composite
\[f:=\psi\circ\sigma|_U\circ\vphi^{-1}:\mb{A}^n_k\to\mb{A}^n_k.\]
The morphism $f$ would vanish at the point $\vphi(p)$. We then define the local index $\ind_p\sigma$ to be the local $\mb{A}^1$-degree of $f$ at $\vphi(p)$. Of course, if $\vphi$ is not an isomorphism, then $\vphi^{-1}$ does not exist. Nevertheless, the assumption that $p$ is an isolated zero of $\sigma$ means that $\mc{O}_{\mb{V}(\sigma),p}$ is a zero-dimensional ring. In particular, the inverse $\vphi^{-1}$ exists up to some power of the ideal corresponding to $p$. Using this, one can show that $\mc{O}_{\mb{V}(\sigma),p}$ is isomorphic as a local ring to \[k[x_1,\ldots,x_n]_{\vphi(p)}/(f_1,\ldots,f_n)\] 
for some polynomial map $f=(f_1,\ldots,f_n):\mb{A}^n_k\to\mb{A}^n_k$~\cite[Lemma 27]{KW21}. The presentation of this local ring canonically determines a bilinear form~\cite{SS75} whose isomorphism class is $\deg_{\vphi(p)}(f)$~\cite{KW19,BBMMO21}.

\begin{defn}\label{def:local index}
The \textit{local index} of $\sigma:X\to V$ at $p$ is
\[\ind_p(\sigma):=\deg_{\vphi(p)}(f)\]
where $f$ is the morphism determined by our choice of Nisnevich coordinates, relative orientation, and compatible local trivialization. By~\cite[Corollary 31]{KW21}, the local index is independent of these choices.
\end{defn}

\section{Coordinates, trivializations, and relative orientability}
Let $X=\mb{P}\Sym^2(\mc{S}^\vee)$ and $E=\mc{O}_X(1)\otimes\pi^*\mc{O}_{\mb{G}(2,3)}(2)$ (where $\pi:X\to\mb{G}(2,3)$ is projection to the base), as described in Section~\ref{sec:approach}. We first prove that $E^{\oplus 8}\to X$ is relatively orientable.

\begin{lem}\label{lem:rel orientable}
The vector bundle $E^{\oplus 8}\to X$ is relatively orientable.
\end{lem}
\begin{proof}
In order to show that $E^{\oplus 8}\to X$ is relatively orientable, we need to show that $\det{E^{\oplus 8}}\otimes\omega_X$ is the tensor square of a line bundle, where $\omega_X$ is the canonical bundle of $X$. Since $E^{\oplus 8}$ is a direct sum of line bundles, we have $\det{E^{\oplus 8}}\cong\mc{O}_X(8)\otimes\pi^*\mc{O}_{\mb{G}(2,3)}(16)$. 

Given a vector bundle $V\to Y$ of rank $r$, the canonical bundle of $\pi:\mb{P}V\to Y$ is given by
\begin{align}\label{eq:canonical bundle}
    \omega_{\mb{P}V}\cong\mc{O}_{\mb{P}V}(-r)\otimes\pi^*\det{V^\vee}\otimes\pi^*\omega_Y.
\end{align}
This can be computed via the short exact sequence of tangent bundles
\[0\to T_{\mb{P}V/Y}\to T_{\mb{P}V}\to\pi^*T_Y\to 0\]
and the tautological exact sequence
\[0\to\mc{O}_{\mb{P}V}(-1)\to\pi^*V\to\mc{Q}\to 0,\]
where $\mc{Q}$ is the tautological quotient bundle of $\mb{P}V\to Y$. 

Before applying Equation~\ref{eq:canonical bundle} to $X=\mb{P}\Sym^2(\mc{S}^\vee)\to\mb{G}(2,3)$, we need to calculate $\det\Sym^2(\mc{S}^\vee)^\vee$ and $\omega_{\mb{G}(2,3)}$. Since $\mb{G}(2,3)\cong\mb{P}^3$, we have $\omega_{\mb{G}(2,3)}\cong\mc{O}_{\mb{G}(2,3)}(-4)$. Recall that if $\mc{E}$ is a vector bundle of rank $r$, then $\det\Sym^n(\mc{E})=(\det\mc{E})^{\otimes\binom{r+n-1}{r}}$. Since $\mc{S}^\vee$ has rank 3 and $\det(\mc{S})\cong\mc{O}_{\mb{G}(2,3)}(1)$, we have $\det\Sym^2(\mc{S}^\vee)^\vee=(\det\mc{S})^{\otimes 4}\cong \mc{O}_{\mb{G}(2,3)}(4)$. Equation~\ref{eq:canonical bundle} thus gives us
\begin{align*}
    \omega_X&\cong\mc{O}_X(-6)\otimes\pi^*\mc{O}_{\mb{G}(2,3)}(4)\otimes\pi^*\mc{O}_{\mb{G}(2,3)}(-4)\\
    &\cong\mc{O}_X(-6).
\end{align*}
Thus $\det E^{\oplus 8}\otimes\omega_X\cong\mc{O}_X(2)\otimes\pi_*\mc{O}_{\mb{G}(2,3)}(16)\cong(\mc{O}_X(1)\otimes\pi_*\mc{O}_{\mb{G}(2,3)}(8))^{\otimes 2}$, as desired.
\end{proof}

Since $E^{\oplus 8}$ is relatively orientable and of rank equal to the dimension of $X=\mb{P}\Sym^2(\mc{S}^\vee)$, this bundle has a well-defined Euler number. We now compute the Euler number $e(E^{\oplus 8})$ using~\cite[Lemma 5 and Proposition 19]{SW21}, which will constitute the fixed global count of conics meeting 8 general lines.

\begin{lem}\label{lem:euler class}
The vector bundle $E^{\oplus 8}\to X$ has Euler number $e(E^{\oplus 8})=46\cdot\mb{H}$.
\end{lem}
\begin{proof}
Apply~\cite[Proposition 19]{SW21} with $\mc{V}:=E^{\oplus 7}$ and $\mc{V}':=E$. It follows that $e(E^{\oplus 8})$ is of the form $n\cdot\mb{H}$, which has rank $2n$. Taking the rank of this Euler number recovers the classical count $\int_X c_1(E)^8=92$~\cite[Lemma 5]{SW21}, so we have $e(E^{\oplus 8})=46\cdot\mb{H}$.
\end{proof}

The Euler number $e(E^{\oplus 8})$ gives us half of our desired enumerative formula. In order to complete this enumerative formula, we need to express $e(E^{\oplus 8})$ as a sum of local contributions (see \cite[Theorem 3]{KW21}). To start, recall the section $\sigma:X\to E^{\oplus 8}$ (described in Section~\ref{sec:approach}) whose vanishing locus corresponds to the set of conics meeting 8 general lines $L_1,\ldots,L_8\subset\mb{P}^3_k$. Our enumerative formula comes from computing the local indices $\ind_{(H,q)}(\sigma)$ in the decomposition
\[e(E^{\oplus 8})=\sum_{(H,q)\in\sigma^{-1}(0)}\ind_{(H,q)}(\sigma).\]

Our next step is to give Nisnevich coordinates for $X$ and local trivializations of $E^{\oplus 8}$ that are compatible with the relative orientation implicit in Lemma~\ref{lem:rel orientable}.

\subsection{Nisnevich coordinates and local trivializations}\label{sec:coords}
In order to compute the local index $\ind_{(H,q)}(\sigma)$ of a conic $(H,q)$ meeting the lines $L_1,\ldots,L_8$, we need to describe Nisnevich coordinates for $X=\mb{P}\Sym^2(\mc{S}^\vee)$ and local trivializations of $E^{\oplus 8}$.

\subsubsection{Coordinates}
Since $X$ is a $\mb{P}^5_k$-bundle over $\mb{G}(2,3)\cong\mb{P}^3_k$, the standard affine covers of $\mb{P}^3_k$ and $\mb{P}^5_k$ yield a convenient affine cover of $X$. However, we will need to slightly modify the standard cover of $\mb{P}^5_k$ for our purposes. Let $U_i=\{x_i\neq 0\}\subset\mb{P}^3_k$. Let $V_j=\{\ell_j\neq 0\}\subset\mb{P}^5_k$, where
\[\ell_j=\begin{cases}
x_j & 0\leq j\leq 2,\\
x_1+x_2+x_3 & j=3,\\
x_0+x_2+x_4 & j=4,\\
x_0+x_1+x_5 & j=5.
\end{cases}\]
Let $u_i:U_i\to\mb{A}^3_k$ and $v_j:V_j\to\mb{A}^5_k$ be given by 
\begin{align*}
\pcoor{x_0:\cdots:x_3}&\mapsto(\tfrac{x_0}{x_i},\ldots,\tfrac{x_{i-1}}{x_i},\tfrac{x_{i+1}}{x_i},\ldots,\tfrac{x_3}{x_i})\quad\text{and}\\
\pcoor{x_0:\cdots:x_5}&\mapsto(\tfrac{x_0}{\ell_j},\ldots,\tfrac{x_{j-1}}{\ell_j},\tfrac{x_{j+1}}{\ell_j},\ldots,\tfrac{x_5}{\ell_j}),
\end{align*}
respectively. Then $u_i\times v_j:U_i\times V_j\to\mb{A}^8_k$ are affine coordinates on $\mb{P}^3_k\times\mb{P}^5_k$. Since $X$ is affine-locally isomorphic to $\mb{P}^3_k\times\mb{P}^5_k$, the affine coordinates $(U_i\times V_j,u_i\times v_j)$ induce affine coordinates $(W_{ij},w_{ij})$ on $X$. Note that $v_j^{-1}$ is given by
\[(\tfrac{x_0}{\ell_j},\ldots,\tfrac{x_{j-1}}{\ell_j},\tfrac{x_{j+1}}{\ell_j},\ldots,\tfrac{x_5}{\ell_j})\mapsto\pcoor{\tfrac{x_0}{x_j}:\cdots:\tfrac{x_{j-1}}{\ell_j}:1-\tfrac{\ell_j-x_j}{\ell_j}:\tfrac{x_{j+1}}{\ell_j}:\cdots:\tfrac{x_5}{\ell_j}}\]
for $3\leq j\leq 5$, where $\frac{\ell_j-x_j}{\ell_j}$ is the sum of two coordinates in $(\tfrac{x_0}{\ell_j},\ldots,\tfrac{x_{j-1}}{\ell_j},\tfrac{x_{j+1}}{\ell_j},\ldots,\tfrac{x_5}{\ell_j})$.

\begin{prop}
The affine coordinates $w_{ij}:W_{ij}\to\mb{A}^8_k$ are Nisnevich coordinates on $X$.
\end{prop}
\begin{proof}
Under the affine-local isomorphism $X\cong\mb{P}^3_k\times\mb{P}^5_k$, it suffices to prove the statement for the affine coordinates $(U_i\times V_j,u_i\times v_j)$. But $u_i:U_i\to\mb{A}^3_k$ and $v_j:V_j\to\mb{A}^5_k$ are isomorphisms and are hence \'etale morphisms that induce isomorphisms of residue fields on closed points.
\end{proof}

Since the coordinates $w_{ij}:W_{ij}\to\mb{A}^8_k$ are isomorphisms, we can also consider the inverse maps $\vphi_{ij}:=w_{ij}^{-1}:\mb{A}^8_k\to W_{ij}$. Using the isomorphisms $\mb{G}(2,3)\cong\mb{P}^3_k$ and $\mb{P}\Sym^2(H^\vee)\cong\mb{P}^5_k$ (where $H\in\mb{G}(2,3)$), we can describe $\vphi_{ij}$ at the level of points by parameterizing homogeneous quadratic polynomials on planes in $\mb{P}^3_k$.

\begin{notn}
Given a polynomial $f$ over $k$, let $[f]$ denote the equivalence class of all $k^\times$ multiples of $f$.
\end{notn}

\begin{notn}\label{notn:coords}
Let $\bm{a}=\pcoor{a_0:a_1:a_2:a_3}\in \mb{P}^3_k$, and let $\bm{b}=\pcoor{b_0:\cdots:b_5}\in \mb{P}^5_k$. Let $H_{\bm{a}}=\mb{V}(\sum_{i=0}^3 a_ix_i)$, which defines the isomorphism $\mb{P}^3_k\to\mb{G}(2,3)$. The isomorphism $\mb{P}^5_k\to\mb{P}\Sym^2(H_{\bm{a}}^\vee)$ is given by choosing coordinates $\pcoor{y_0:y_1:y_2}$ on $H_\bm{a}$ and assigning
\[\bm{b}\mapsto q_{\bm{b}}(y_0,y_1,y_2):=[b_0y_0^2+b_1y_1^2+b_2y_2^2+b_3y_1y_2+b_4y_0y_2+b_5y_0y_1].\]
\end{notn}

Since there is no canonical dual basis for a given plane, there is no canonical isomorphism $\mb{P}^5_k\to\mb{P}\Sym^2(H^\vee)$. In defining $\vphi_{ij}:\mb{A}^8_k\to W_{ij}$, we choose the isomorphism $\mb{P}^5_k\to\mb{P}\Sym^2(H^\vee)$ that conforms with a particular choice of dual bases for the planes in $\pi(W_{ij})\subset\mb{G}(2,3)$.

\begin{prop}\label{prop:coords}
Let $0\leq i\leq 3$ and $0\leq j\leq 5$. Let $\alpha_i=u_i^{-1}$ and $\beta_j=v_j^{-1}$. Then on $k$-points, $\vphi_{ij}:\mb{A}^8_k\to X$ is given by 
\[(a_1,a_2,a_3,b_1,\cdots,b_5)\mapsto(H_{\alpha_i(a_1,a_2,a_3)},q_{\beta_j(b_1,\cdots,b_5)}(y_{i0},y_{i1},y_{i2})),\] 
where $\pcoor{y_{i0}:y_{i1}:y_{i2}}=\pcoor{x_0:\ldots:\hat{x}_i:\ldots:x_3}$.
\end{prop}
\begin{proof}
Since $H_{\bm{a}}:\mb{P}^3_k\to\mb{G}(2,3)$ and $q_{\bm{b}}:\mb{P}^5_k\to\mb{P}\Sym^2(H^\vee)$ yield the desired isomorphisms, it remains to treat the choice of coordinates on $H$ as this plane changes. The coordinates $\pcoor{y_{i0}:y_{i1}:y_{i2}}$ are the standard coordinates of the plane $\mb{V}(x_i)\subset\mb{P}^3_k$. Since the plane $H_{\alpha_i(a_1,a_2,a_3)}$ surjects onto $\mb{V}(x_i)$ under projection, $\pcoor{y_{i0}:y_{i1}:y_{i2}}$ are indeed projective coordinates on $H_{\alpha_i(a_1,a_2,a_3)}$. By construction, our choice of coordinates on $H_{\alpha_i(a_1,a_2,a_3)}$ is constant on $W_{ij}$.
\end{proof}

\begin{rem}
Since the Grassmannian $\mb{G}(2,3)$ and the space of conics $X$ respect base change, Proposition~\ref{prop:coords} also describes $\vphi_{ij}$ on $k'$-points for any finite extension $k'$ of $k$.
\end{rem}

\subsubsection{Trivializations}
Next, we give local trivializations $\psi_{ij}:E^{\oplus 8}|_{W_{ij}}\to W_{ij}\times\mb{A}^8_k$. We do this by describing local trivializations $\psi'_{ij}:E|_{W_{ij}}\to W_{ij}\times\mb{A}^1_k$ and setting $\psi_{ij}=\bigoplus_{\ell=1}^8\psi'_{ij}$. (Later, we will conflate $\psi_{ij}$ with the composite $E^{\oplus 8}|_{W_{ij}}\to W_{ij}\times\mb{A}^8_k\to\mb{A}^8_k$.) To define $\psi'_{ij}$, it suffices to construct a non-vanishing section $\tau_{ij}:W_{ij}\to E|_{W_{ij}}$. We will accomplish this by choosing a line $T_{ij}$ such that the evaluation section $\mr{ev}_{T_{ij}}\circ z:W_{ij}\to E|_{W_{ij}}$ is non-vanishing.

\begin{defn}
Let $\pcoor{y_{i0}:y_{i1}:y_{i2}}=\pcoor{x_0:\cdots:\hat{x}_i:\cdots:x_3}$. For $0\leq i\leq 3$ and $0\leq j\leq 2$, let $T_{ij}:=\mb{V}(\{y_{i\ell}\}_{\ell\neq j})$. For $0\leq i\leq 3$ and $3\leq j\leq 5$, let $T_{ij}:=\mb{V}(y_{i,j-3},y_{im}-y_{in})$, where $m<n$ and $\{j-3,m,n\}=\{0,1,2\}$. These lines are chosen so that
\[T_{ij}\cap H_{\bm{a}}=\begin{cases}
\pcoor{-a_1:a_0:0:0} & (i,j)=(0,0),\\
\pcoor{-a_2:0:a_0:0} & (i,j)=(0,1),\\
\pcoor{-a_3:0:0:a_0} & (i,j)=(0,2),\\
 &\vdots\\
\pcoor{a_3:0:0:-a_0} & (i,j)=(3,0),\\
\pcoor{0:a_3:0:-a_1} & (i,j)=(3,1),\\
\pcoor{0:0:a_3:-a_2} & (i,j)=(3,2),
\end{cases}
\quad
\begin{cases}
\pcoor{-a_2-a_3:0:a_0:a_0} & (i,j)=(0,3),\\
\pcoor{-a_1-a_3:a_0:0:a_0} & (i,j)=(0,4),\\
\pcoor{-a_1-a_2:a_0:a_0:0} & (i,j)=(0,5),\\
 &\vdots\\
\pcoor{a_3:a_3:0:-a_0-a_1} & (i,j)=(3,3),\\
\pcoor{a_3:0:a_3:-a_0-a_2} & (i,j)=(3,4),\\
\pcoor{0:a_3:a_3:-a_1-a_2} & (i,j)=(3,5).
\end{cases}\]
Next, let $\tau_{ij}:=\mr{ev}_{T_{ij}}\circ z:W_{ij}\to E|_{W_{ij}}$ (see Section~\ref{sec:approach}). On points, this section is given by
\[\tau_{ij}(H_{\bm{a}},q_{\bm{b}})=((H_\bm{a},q_\bm{b}),q_\bm{b}\text{ mod }Q(H_\bm{a})_{T_{ij}\cap H_\bm{a}}),\]
where $Q(H_\bm{a})_{T_{ij}\cap H_\bm{a}}$ is the space of homogeneous quadratic polynomials on $H_\bm{a}$ that vanish on $T_{ij}\cap H_\bm{a}$. By construction, $q_\bm{b}$ does not vanish on $T_{ij}\cap H_\bm{a}$ for $(H_\bm{a},q_\bm{b})\in W_{ij}$, since
\[q_{\bm{b}}(T_{ij}\cap H_{\bm{a}})=\begin{cases}
a_i^2b_j & 0\leq j\leq 2,\\
a_i^2(b_1+b_2+b_3) & j=3,\\
a_i^2(b_0+b_2+b_4) & j=4,\\
a_i^2(b_0+b_1+b_5) & j=5
\end{cases}\]
and
\[Q(H_\bm{a})_{T_{ij}\cap H_{\bm{a}}}\cong\begin{cases}
\{b_j=0\} & 0\leq j\leq 2,\\
\{b_1+b_2+b_3= 0\} & j=3,\\
\{b_0+b_2+b_4= 0\} & j=4,\\
\{b_0+b_1+b_5= 0\} & j=5.
\end{cases}\]
Thus $\tau_{ij}$ is a non-vanishing section on $W_{ij}$ and hence determines a local trivialization of $E|_{W_{ij}}$.
\end{defn}

\subsection{Compatibility}
Using our Nisnevich coordinates $w_{ij}=\vphi_{ij}^{-1}$ and local trivializations $\psi_{ij}$, we will compute the local index $\ind_{(H,q)}(\sigma)$ in terms of the local $\mb{A}^1$-degree $\deg^{\mb{A}^1}_{w_{ij}(H,q)}(\psi_{ij}\circ\sigma\circ\vphi_{ij})$ via Definition~\ref{def:local index}. In order to do this, we need to certify that the trivialization $\psi_{ij}$ is compatible with the Nisnevich coordinates $w_{ij}$ and the relative orientation
\[\Hom(\det (E^{\oplus 8})^\vee,\omega_X)|_{W_{ij}}\cong(\mc{O}_X(1)\otimes\pi^*\mc{O}_{\mb{G}(2,3)}(8))^{\otimes 2}|_{W_{ij}}\]
from Lemma~\ref{lem:rel orientable}. First, let $\det dw_{ij}\in H^0(W_{ij},\omega_X)$ be the non-vanishing section determined by our Nisnevich coordinates on $W_{ij}$. Let
\begin{align*}
    \gamma_{ij}:\Hom(\det (E^{\oplus 8})^\vee,\omega_X)|_{W_{ij}}\xrightarrow{\cong}\Hom(\det (E^{\oplus 8})^\vee,\mc{O}_X(-6))|_{W_{ij}}
\end{align*}
be the isomorphism induced by the isomorphism $\omega_X|_{W_{ij}}\cong\mc{O}_X(-6)|_{W_{ij}}$ sending $\det dw_{ij}$ to the non-vanishing local section $(a_i,\ell_j)^{-6}\in H^0(W_{ij},\mc{O}_X(-6))$. Let
\begin{align*}
\delta_{ij}:\Hom(\det (E^{\oplus 8})^\vee,\mc{O}_X(-6))|_{W_{ij}}&\xrightarrow{\cong}\mc{O}_X(8)\otimes\pi^*\mc{O}_{\mb{G}(2,3)}(16)\otimes\mc{O}_X(-6)|_{W_{ij}}\\
&\xrightarrow{\cong}(\mc{O}_X(1)\otimes\pi^*\mc{O}_{\mb{G}(2,3)}(8))^{\otimes 2}|_{W_{ij}}
\end{align*}
be the isomorphism given by sending
\begin{align*}
[\det\psi_{ij}^\vee\mapsto(a_i,\ell_j)^{-6}]&\mapsto(a_i,\ell_j)^8\otimes\pi^*(a_i^2)^8\otimes(a_i,\ell_j)^{-6}\\
&\mapsto((a_i,\ell_j)\otimes\pi^*(a_i)^8)^{\otimes 2}.
\end{align*}

\begin{prop}\label{prop:compatible}
The local trivializations $\psi_{ij}$ are compatible with the Nisnevich coordinates $w_{ij}$ and the relative orientation
\[\varpi_{ij}:=\delta_{ij}\circ\gamma_{ij}:\Hom(\det E^{\oplus 8},\omega_X)|_{W_{ij}}\xrightarrow{\cong}(\mc{O}_X(1)\otimes\pi^*\mc{O}_{\mb{G}(2,3)}(8))^{\otimes 2}|_{W_{ij}}.\]
\end{prop}
\begin{proof}
By construction, we have $\varpi_{ij}([\det\psi_{ij}^\vee\mapsto\det dw_{ij}])=((a_i,\ell_j)\otimes\pi^*(a_i)^8)^{\otimes 2}$, which is a tensor square as desired.
\end{proof}

\section{Local contributions}\label{sec:geometric interpretation}
We now compute the local index $\ind_{(H,q)}(\sigma)$ of a conic $(H,q)$ meeting the lines $L_1,\ldots,L_8$. By \cite[Definition 30 and Corollary 31]{KW21}, the local index is equal to the local $\mb{A}^1$-degree $\deg^{\mb{A}^1}_{w_{ij}(H,q)}(\Phi_{ij})$ of the composite
\[\Phi_{ij}:=\psi_{ij}\circ\sigma\circ\vphi_{ij}\]
for any $i,j$ such that $(H,q)\in U_{ij}$. (Here, $\psi_{ij}$ is the local trivialization $E^{\oplus 8}|_{W_{ij}}\to W_{ij}\times\mb{A}^8_k$ post-composed with the projection $W_{ij}\times\mb{A}^8_k\to\mb{A}^8_k$.) By~\cite[Proposition 9.25]{eisenbud_harris_2016}, all zeros of $\Phi_{ij}$ are simple, so local degree can be computed by the Jacobian determinant of $\Phi_{ij}$ evaluated at $w_{ij}(H,q)$ \cite[Lemma 9]{KW19}\label{prop:ind}. Since $k$ is assumed to be perfect, the field of definition $k(q)$ of $(H,q)$ is separable over $k$, so~\cite[Proposition 34]{KW21} implies that we can compute the local index by base changing and post-composing with the field trace $\Tr_{k(q)/k}$ (see also~\cite[Theorem 1.3]{BBMMO21}). As a result, we have the following proposition.

\begin{prop}
Let $k(q)$ be the field of definition of $(H,q)\in W_{ij}$, let $\Tr_{k(q)/k}$ be the field trace, and let $\Jac(\Phi_{ij})$ be the determinant of the Jacobian matrix of $\Phi_{ij}$. Then
\[\ind_{(H,q)}(\sigma)=\Tr_{k(q)/k}\langle\Jac(\Phi_{ij})|_{w_{ij}(H,q)}\rangle.\]
\end{prop}

Our goal in this section is to interpret $\Jac(\Phi_{ij})|_{w_{ij}(H,q)}$ in terms of geometric information intrinsic to the conic $(H,q)$ and the lines $L_1,\ldots,L_8$. To simplify notation, let $\Phi_{ij}^n=\psi_{ij}'\circ\sigma_{L_n}\circ\vphi_{ij}$, so that $\Phi_{ij}=(\Phi_{ij}^0,\ldots,\Phi_{ij}^8)$. On points (which is the generality needed for our computations), the evaluation section can be computed as $\mr{ev}_L\circ z(H,q)=q(L\cap H)$. As a function of $(a_1,a_2,a_3,b_1,\ldots,b_5)\in\mb{A}^8_k$, we can thus write out $\Phi^n_{ij}$:
\begin{align*}
    \Phi^n_{ij}(a_1,\ldots,b_5)&=\frac{\sigma_n(H_{\alpha_i(a_1,a_2,a_3)},q_{\beta_j(b_1,\ldots,b_5)})}{\tau_{ij}(H_{\alpha_i(a_1,a_2,a_3)},q_{\beta_j(b_1,\ldots,b_5)})}\\
    &=\frac{q_{\beta_j(b_1,\ldots,b_5)}(L_n\cap H_{\alpha_i(a_1,a_2,a_3)})}{q_{\beta_j(b_1,\ldots,b_5)}(T_{ij}\cap H_{\alpha_i(a_1,a_2,a_3)})}\\
    &=\begin{cases}
   z_1^2+b_1z_2^2+b_2z_3^2+b_3z_2z_3+b_4z_1z_3+b_5z_1z_2 & j=0,\\
   b_1z_1^2+z_2^2+b_2z_3^2+b_3z_2z_3+b_4z_1z_3+b_5z_1z_2 & j=1,\\
   b_1z_1^2+b_2z_2^2+z_3^2+b_3z_2z_3+b_4z_1z_3+b_5z_1z_2 & j=2,\\
   b_1z_1^2+b_2z_2^2+b_3z_3^2+(1-b_2-b_3)z_2z_3+b_4z_1z_3+b_5z_1z_2 & j=3,\\
   b_1z_1^2+b_2z_2^2+b_3z_3^2+b_4z_2z_3+(1-b_1-b_3)z_1z_3+b_5z_1z_2 & j=4,\\
   b_1z_1^2+b_2z_2^2+b_3z_3^2+b_4z_2z_3+b_5z_1z_3+(1-b_1-b_2)z_1z_2 & j=5,\\
    \end{cases}\stepcounter{equation}\tag{\theequation}\label{eq:Phi}
\end{align*}
where $(z_1,z_2,z_3)$ is a preferred affine representative of the point $L_n\cap H_{\alpha_i(a_1,a_2,a_3)}$. While the coordinates of this point are a priori only defined up to a scalar, part of the trivialization data is a dehomogenization of these coordinates. However, we will not need to work this out explicitly since the chain rule implies that the partial derivatives $\frac{\partial\Phi^n_{ij}}{\partial a_\ell}$ will be independent of scaling $(z_1,z_2,z_3)$.

\begin{rem}
Recall that each $q_\bm{b}$ is not simply a quadratic polynomial, but a projective class of such polynomials. The scalar action on these classes is cancelled out in the ratio in Equation~\ref{eq:Phi}. Similarly, the points $L_n\cap H_\bm{a}$ and $T_{ij}\cap H_\bm{a}$ live in the projective plane, so their coordinates are only defined up to scalars. Again, the trivialization data cancels out the scalar action on the coordinates of these points, so we get a well-defined function.
\end{rem}

\subsection{Geometric interpretation}
In order to provide a geometric interpretation of $\Jac(\Phi_{ij})|_{w_{ij}(H,q)}$, we will give a geometric interpretation of the partial derivatives $\frac{\partial\Phi^n_{ij}}{\partial a_\ell}$ and $\frac{\partial\Phi^n_{ij}}{\partial b_\ell}$. We will phrase this interpretation in terms of the affine geometry underlying our projective conics and projective lines; in particular, $H_{\alpha_i(a_1,a_2,a_3)}$ will be an affine 3-plane instead of a projective 2-plane, $L_n$ will be an affine 2-plane instead of a projective line, and $\mb{V}(q_{\beta_j(b_1,\ldots,b_5)})$ will be an affine cone instead of a projective conic. We start with $\frac{\partial\Phi^n_{ij}}{\partial b_\ell}$. Since $\Phi^n_{ij}$ are linear in $b_\ell$, these partial derivatives are straightforward to compute (see Figure~\ref{fig:dPhi/db}). Geometrically, the entries $\frac{\partial\Phi^n_{ij}}{\partial b_\ell}$ of the Jacobian matrix of $\Phi_{ij}$ at $w_{ij}(H,q)$ are recording the coordinates of the intersections $L_n\cap H_{\alpha_i(a_1,a_2,a_3)}$ in terms of the coordinates $(y_{i0},y_{i1},y_{i2})$ on the plane $H_{\alpha_i(a_1,a_2,a_3)}$. These coordinates are slightly modified when $3\leq j\leq 5$ and $1\leq\ell\leq 3$, due to our modified coordinates $w_{ij}$ when $3\leq j\leq 5$.

\begin{figure}[h]
 \begin{tabular}{|c | c c c c c c|} 
 \hline
 \tiny{\diagbox{$\ell$}{$j$}} & 0 & 1 & 2 & 3 & 4& 5\\ 
 \hline
 1 & $z_2^2$ & $z_1^2$ & $z_1^2$ & $z_1^2$ & $z_1^2-z_1z_3$ & $z_1^2-z_1z_2$ \\
 2 & $z_3^2$ & $z_3^2$ & $z_2^2$ & $z_2^2-z_2z_3$ & $z_2^2$ & $z_2^2-z_1z_2$ \\
 3 & $z_2z_3$ & $z_2z_3$ & $z_2z_3$ & $z_3^2-z_2z_3$ & $z_3^2-z_1z_3$ & $z_3^2$\\
 4 & $z_1z_3$ & $z_1z_3$ & $z_1z_3$ & $z_1z_3$ & $z_2z_3$ & $z_2z_3$ \\
 5 & $z_1z_2$ & $z_1z_2$ & $z_1z_2$ & $z_1z_2$ & $z_1z_2$ & $z_1z_3$ \\
 \hline
\end{tabular}
\caption{$\frac{\partial\Phi^n_{ij}}{\partial b_\ell}$}
\label{fig:dPhi/db}
\end{figure}

Now we consider $\frac{\partial\Phi^n_{ij}}{\partial a_\ell}$. Noting that $z_1,z_2,z_3$ are functions of $(a_1,a_2,a_3)$, we compute $\frac{\partial\Phi^n_{ij}}{\partial a_\ell}$ via the chain rule:
\begin{align}\label{eq:dPhi/da}
    \frac{\partial\Phi^n_{ij}}{\partial a_\ell}&=\frac{\partial\Phi^n_{ij}}{\partial z_1}\cdot\frac{\partial z_1}{\partial a_\ell}+\frac{\partial\Phi^n_{ij}}{\partial z_2}\cdot\frac{\partial z_2}{\partial a_\ell}+\frac{\partial\Phi^n_{ij}}{\partial z_3}\cdot\frac{\partial z_3}{\partial a_\ell}.
\end{align}
Equation~\ref{eq:dPhi/da} is closely related to the tangent plane in $H_{\alpha_i(a_1,a_2,a_3)}\cong\mb{A}^3_{k(q)}$ of $\mb{V}(q_{\beta_j(b_1,\ldots,b_5)})$ at $(z_1,z_2,z_3)$. Indeed, replacing $z_1,z_2,z_3$ with the projective coordinates $y_{i0},y_{i1},y_{i2}\in H_{\alpha_i(a_1,a_2,a_3)}^\vee$ in Equation~\ref{eq:Phi}, we let
\[q_{ij}:=q_{\beta_j(b_1,\ldots,b_5)}=\begin{cases}
    z_{i0}^2+b_1z_{i1}^2+b_2z_{i2}^2+b_3z_{i1}z_{i2}+b_4z_{i0}z_{i2}+b_5z_{i0}z_{i1} & j=0,\\
    b_1z_{i0}^2+z_{i1}^2+b_2z_{i2}^2+b_3z_{i1}z_{i2}+b_4z_{i0}z_{i2}+b_5z_{i0}z_{i1} & j=1,\\
    & \vdots\\
    b_1z_{i0}^2+b_2z_{i1}^2+b_3z_{i2}^2+b_4z_{i1}z_{i2}+b_5z_{i0}z_{i2}+(1-b_1-b_2)z_{i0}z_{i1} & j=5\\
    \end{cases}\]
be the defining equation for $\mb{V}(q_{\beta_j(b_1,\ldots,b_5)})\subset H_{\alpha_i(a_1,a_2,a_3)}\cong\mb{A}^3_{k(q)}$ (see Figure~\ref{fig:cone1}). Any conic meeting 8 general lines in $\mb{P}^3$ is smooth by \cite[Lemma 9.21]{eisenbud_harris_2016}, so the affine cone $\mb{V}(q_{ij})$ (corresponding to the projective conic $\mb{V}(q_{ij})$) is smooth away from the cone point at the origin $(0,0,0)\in\mb{A}^3_{k(q)}$. It follows that the tangent plane $T_p\mb{V}(q_{ij})$ at $p:=(z_1,z_2,z_3)$ is defined by the equation
\[\frac{\partial q_{ij}}{\partial y_{i0}}\bigg|_p\cdot y_{i0}+\frac{\partial q_{ij}}{\partial y_{i1}}\bigg|_p\cdot y_{i1}+\frac{\partial q_{ij}}{\partial y_{i2}}\bigg|_p\cdot y_{i2}=0.\]
Since $\frac{\partial\Phi^n_{ij}}{\partial z_\ell}=\frac{\partial q_{ij}}{\partial y_{i,\ell-1}}|_p$ and $p$ lies on the tangent plane $T_p\mb{V}(q_{ij})$, we have
\begin{align*}
\frac{\partial\Phi^n_{ij}}{\partial a_\ell}&=\frac{\partial q_{ij}}{\partial y_{i0}}\bigg|_p\left(\frac{\partial z_1}{\partial a_\ell}\right)+\frac{\partial q_{ij}}{\partial y_{i1}}\bigg|_p\left(\frac{\partial z_2}{\partial a_\ell}\right)+\frac{\partial q_{ij}}{\partial y_{i2}}\bigg|_p\left(\frac{\partial z_3}{\partial a_\ell}\right)\\
&=\frac{\partial q_{ij}}{\partial y_{i0}}\bigg|_p\left(z_1+\frac{\partial z_1}{\partial a_\ell}\right)+\frac{\partial q_{ij}}{\partial y_{i1}}\bigg|_p\left(z_2+\frac{\partial z_2}{\partial a_\ell}\right)+\frac{\partial q_{ij}}{\partial y_{i2}}\bigg|_p\left(z_3+\frac{\partial z_3}{\partial a_\ell}\right).\stepcounter{equation}\tag{\theequation}\label{eq:level set}
\end{align*}
Geometrically, this records information about the slope of the affine plane $L_n$ (corresponding to the projective line $L_n$) relative to the plane $T_p\mb{V}(q_{ij})$. Indeed, $z_m+\frac{\partial z_m}{\partial a_\ell}$ is $z_m$ shifted by the rate of change of $z_m$ (representing a coordinate of a basis vector in $L_n\cap H_{\alpha_i(a_1,a_2,a_3)}$) as $a_{\ell}$ changes. Equation~\ref{eq:level set} states that $(z_1+\frac{\partial z_1}{\partial a_\ell},z_2+\frac{\partial z_2}{\partial a_\ell},z_3+\frac{\partial z_3}{\partial a_\ell})$ lies on the level set defined by
\[\frac{\partial q_{ij}}{\partial y_{i0}}\bigg|_p\cdot y_{i0}+\frac{\partial q_{ij}}{\partial y_{i1}}\bigg|_p\cdot y_{i1}+\frac{\partial q_{ij}}{\partial y_{i2}}\bigg|_p\cdot y_{i2}=\frac{\partial\Phi^n_{ij}}{\partial a_\ell},\]
which is parallel to the tangent plane $T_p\mb{V}(q_{ij})$. It follows that $\frac{\partial\Phi^n_{ij}}{\partial a_\ell}$ measures the deviation of $p+\frac{\partial p}{\partial a_\ell}$ from the tangent plane $T_p\mb{V}(q_{ij})$, as illustrated in Figure~\ref{fig:cone2}.

\begin{figure}
    \centering
    \begin{tikzpicture}[tdplot_main_coords,scale=3/2]
    \coordinate (O) at (0,0,0);
    \coneback[surface,draw=cyan,dotted]{2}{4/3}{10}
    \conefront[surface,draw=cyan]{2}{4/3}{10}
    \draw[draw=red,thick] (2,0,3) -- (0,0,0);
    \coneback[surface]{3}{2}{10}
    \conefront[surface]{3}{2}{10}
    \fill (4/3,0,2) circle[radius=1pt];
    \node at (1/2,1/2,2-1/3) {$p$};
    \node[text=red] at (0,2,1/2) {$L_n\cap H_{\alpha_i(a_1,a_2,a_3)}$};
    \node at (0,-2,3/2) {$\mb{V}(q_{ij})$};
    \end{tikzpicture}
    \caption{$\mb{V}(q_{ij})$ as a cone in $\mb{A}^3_{k(q)}$}
    \label{fig:cone1}
\end{figure}

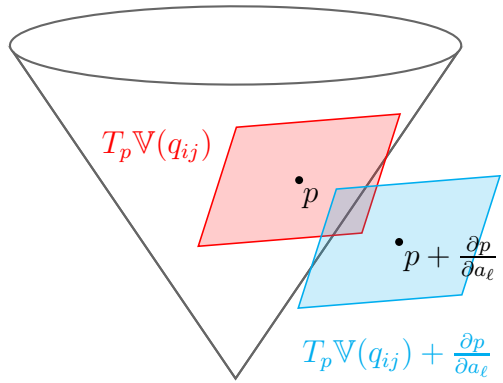
\begin{figure}
    \centering
    \begin{tikzpicture}[tdplot_main_coords,scale=3/2,tangent plane/.style args={at #1 with vectors #2 and #3}{%
insert path={($#1+($#2-(0,0,0)$)+($#3-(0,0,0)$)$) --  ($#1+($#2-(0,0,0)$)-($#3-(0,0,0)$)$) --  ($#1-($#2-(0,0,0)$)-($#3-(0,0,0)$)$) 
-- ($#1-($#2-(0,0,0)$)+($#3-(0,0,0)$)$) -- cycle}}]
    \coordinate (O) at (0,0,0);
    \coneback[surface]{3}{2}{10}
    \conefront[surface]{3}{2}{10}
    \draw[draw=red,semithick,fill=red,fill opacity=0.2,tangent plane=at {(4/3,0,2)} with vectors {(2/5,0,3/5)} and {(0,4/5,0)}];
    \fill (4/3,0,2) circle[radius=1pt];
    \node at (1/2,1/2,2-1/3) {$p$};
    \draw[draw=cyan,semithick,fill=cyan,fill opacity=0.2,tangent plane=at {(2,2/3,3/2)} with vectors {(2/5,0,3/5)} and {(0,4/5,0)}];
    \fill (2,2/3,3/2) circle[radius=1pt];
    \node at (2-1/3,1+1/3,3/2-1/4) {$p+\frac{\partial p}{\partial a_\ell}$};
    \node[text=red] at (1/2,-1,2+1/4) {$T_p\mb{V}(q_{ij})$};
    \node[text=cyan] at (1+1/4,1,1/3) {$T_p\mb{V}(q_{ij})+\frac{\partial p}{\partial a_\ell}$};
    \end{tikzpicture}
    \caption{$T_p\mb{V}(q_{ij})$ and its $\frac{\partial\Phi^n_{ij}}{\partial a_\ell}$-level set}
    \label{fig:cone2}
\end{figure}

\begin{rem}
Like many problems in enumerative geometry, the problem of counting conics through eight lines in $\mb{P}^3$ is \textit{of lci type}. That is, sections of the vector bundle encoding this enumerative problem are local complete intersection morphisms. The local indices of such problems can always be interpreted as an \textit{intersection volume} as introduced in~\cite{McK21} (see also~\cite[Section 5.1]{McK22}). 

However, this intersection volume describes the geometry of certain hypersurfaces in the parameter space of objects being counted --- in the present case, these hypersurfaces sit within $\mb{P}\Sym^2(\mc{S}^\vee)$. Generally, this geometry is a step removed from the objects that we actually want to consider (i.e.~the conics themselves). In this article, we have translated this intersection volume into terms pertaining to the conics being parameterized. Finding other, more parsimonious geometric descriptions for the local indices of this article would constitute interesting progress on the geometricity problem~\cite[Question 5.1 and Appendix C]{McK22}.
\end{rem}

\section{Real conics meeting 8 lines}\label{sec:real}
Let $k=\mb{R}$. If a non-real conic meets 8 general lines in $\mb{P}^3_\mb{R}$, then its complex conjugate meets these 8 lines as well. In particular, the number of real conics meeting the 8 lines must be an even integer between 0 and 92. By Hauenstein--Sottile \cite[Table 6]{HS12} and Griffin--Hauenstein \cite[Theorem 1]{GH15}, there exist 8 general lines for each $2n\in\{0,2,\ldots,92\}$ realizing the count of $2n$ real conics. It follows from Theorem~\ref{thm:main-enriched} that these $2n$ conics come in two families of order $n$.

\begin{defn}\label{def:pos/neg}
Let $L_1,\ldots,L_8$ be 8 general lines in $\mb{P}^3_\mb{R}$. A real conic $(H,q)$ meeting $L_1,\ldots,L_8$ is called \textit{positive} (respectively, \textit{negative}) if $\Jac(\Phi_{ij})|_{w_{ij}(H,q)}$ is positive (respectively, negative). Note that this definition implicitly depends on the order of $L_1,\ldots,L_8$; permuting these lines by an odd permutation turns a positive conic into a negative conic (and vice versa).
\end{defn}

\begin{rem}
By \cite[Proposition 9.25]{eisenbud_harris_2016}, all zeros of $\Phi_{ij}$ are simple and hence 
\[\Jac(\Phi_{ij})|_{w_{ij}(H,q)}\neq 0.\]
Moreover, since we have assumed that $(H,q)$ is real, it follows that $\Jac(\Phi_{ij})|_{w_{ij}(H,q)}$ is real. Finally, Proposition~\ref{prop:compatible} implies that $\Jac(\Phi_{ij})|_{w_{ij}(H,q)}$ depends on the choice of $W_{ij}\ni(H,q)$ only up to squares in $\mb{R}$. In particular, the sign of this value does not depend on the choice of $W_{ij}$, so Definition~\ref{def:pos/neg} is well-defined.
\end{rem}

\begin{thm}\label{thm:real}
Given 8 general lines $L_1,\ldots,L_8$ in $\mb{P}^3_\mb{R}$, there are an equal number of positive and negative real conics meeting $L_1,\ldots,L_8$.
\end{thm}
\begin{proof}
The result follows from Theorem~\ref{thm:main-enriched} by taking the signature of
\[46\cdot\mb{H}=\sum_{(H,q)\in\sigma^{-1}(0)}\Tr_{\mb{R}(q)/\mb{R}}\langle\Jac(\Phi_{ij})|_{w_{ij}(H,q)}\rangle.\]
The signature of $\Tr_{\mb{C}/\mb{R}}\langle c\rangle$ is 0 for any $c\in\mb{C}^\times$, and the signature of $46\cdot\mb{H}$ is likewise 0. We thus have
\begin{align*}
    0&=\sum_{\substack{\text{real positive}\\(H,q)}}1+\sum_{\substack{\text{real negative}\\(H,q)}}(-1)\\
    &=\#\{\text{positive real conics}\}-\#\{\text{negative real conics}\},
\end{align*}
as desired.
\end{proof}

\bibliography{conics-meeting-lines}{}
\bibliographystyle{alpha}
\end{document}